\documentclass[twoside]{irmaems}
\usepackage[applemac]{inputenc}
\usepackage{makeidx}

\usepackage{amsmath,amssymb,enumerate,color}

\usepackage{epsfig,fancyhdr,color}
\usepackage{psfrag}
\usepackage{graphicx}
\usepackage{pstricks-add}

\numberwithin{equation}{section}   
\renewcommand{\r}{\mathbb{R}}
\newcommand{\tq}{\, \big| \, }
\DeclareMathOperator{\g}{g} 
\DeclareMathOperator{\Iso}{Iso}

\makeatletter

\definecolor{NoteColor}{rgb}{1,0,0}

\newtheorem{theorem}{\rm\bf Theorem}[section]
\newtheorem{proposition}[theorem]{\rm\bf Proposition}
\newtheorem{lemma}[theorem]{\rm\bf Lemma}
\newtheorem{corollary}[theorem]{\rm\bf Corollary}

\theoremstyle{definition}
\newtheorem{definition}[theorem]{\rm\bf Definition}

\theoremstyle{remark}
\newtheorem{remark}[theorem]{\rm\bf Remark}
\newtheorem{example}[theorem]{\rm\bf Example}

\title{Weak Minkowski Spaces}
\date{today}

\author{Athanase Papadopoulos\thanks{The first author is partially supported by  the French ANR project FINSLER.} and Marc Troyanov} 
 
\address{A. Papadopoulos
 Institut de Recherche Mathematique Avanc\'ee, \\
 Universit\'e de Strasbourg and CNRS, \\
 7 rue Rene Descartes, 67084 Strasbourg Cedex, France.
\\ M. Troyanov, 
Section de Math\'ematiques,  \\
\'Ecole Polytechnique F{\'e}derale de Lausanne, 
1015 Lausanne - Switzerland}

\date{\today}

\begin{document}

\maketitle
 
\begin{abstract}
 We define the notion of weak Minkowski metric and prove some basic properties
of such metrics. We also highlight some of the important analogies between Minkowski geometry and the Funk and Hilbert geometries.

 \medskip
 
\noindent AMS classification:  51B20; 53C70; 53C22; 53C60.

 \medskip

\noindent Keywords:  Weak Minkowski space; Minkowski geometry; norm; Hilbert geometry; Funk geometry; weak norm; Mazur-Ulam theorem; Desarguesian space; Busemann G-space.

\end{abstract}

\tableofcontents

\section{Introduction}

In the last decade of the $19^{th}$ century, Hermann Minkowski initiated new geometric methods in number theory, which culminated
with the celebrated \emph{Geometrie der Zahlen}
 published in 1910 \cite{Minkowski}.  Minkowski's work  is referred to several times by David Hilbert in his 1900 ICM lecture \cite{Hilbert-Problems}, in particular in the introduction, where he declares:
 
\begin{quote}\small
The agreement between geometrical and arithmetical thought is shown also in that we do not habitually follow the chain of reasoning back to the axioms in arithmetical, any more than in geometrical discussions. On the contrary we apply, especially in first attacking a problem, a rapid, unconscious, not absolutely sure combination, trusting to a certain arithmetical feeling for the behavior of the arithmetical symbols, which we could dispense with as little in arithmetic as with the geometrical imagination in geometry. As an example of an arithmetical theory operating rigorously with geometrical ideas and signs, I may mention Minkowski's work, \emph{Die Geometrie der Zahlen}.
\end{quote}

Regarding the influence of this book on the birth of metric geometry, let us mention the following, from the paper \cite{BP1993} by Busemann and Phadke, p. 181: 

\begin{quote}\small
Busemann had read the beginning of Minkowski's \emph{Geometrie der Zahlen} in 1926 which convinced him of the importance of non-Riemannian metrics.
\end{quote}

 An early result of Minkowski in that theory, related to number theory,  states that  any convex domain in $\r^2$ which
is  symmetric around the origin and has area   greater than four contains at least one non-zero point with integer coordinates.  
One step in Minkowski's proof amounts to considering a metric on the plane for which the unit ball at any point in $\r^2$ (in fact in $\mathbb{Z}^2$) is
a translate of the initial convex domain. Such a metric is not Euclidean, it is translation invariant and the straight Euclidean lines are shortest paths. The geometric study of this type of metrics is called  (since Hilbert's writings) \emph{Minkowski Geometry}. We refer to \cite{Busemann1955,Day,Martini2001,Martini2007,Thompson} 
for general expositions of the subject. Minkowski formulated the basic principles of this geometry  in his 1896 paper \cite{Minkowski1896}, and these principles are recalled in Hilbert's lecture \cite{Hilbert-Problems}  (Problem IV).

An express description of Minkowski geometry is the following: choose a convex set $\Omega$ in $\r^n$ that contains the origin. For points $p$
and $q$ in $\mathbb{R}^n$, define a number $\delta >0$ as follows. First dilate $\Omega$ by the factor $\delta$ and then translate the set
in such a way that $0$ is sent to $p$ and $q$ lies on the boundary of the resulting set. In other word $\delta$ is defined by the condition
\begin{equation}\label{def.mk}
 q \in \partial (p + \delta \cdot \Omega).
\end{equation}
 
We denote by $\delta (p,q)$ the number defined in this way. The function $\delta : \r^n \times  \r^n \to \r_+$ is what we call a \emph{weak metric}.
It satisfies the triangle inequality and $\delta(p,p) = 0$. It is not symmetric in general and it can be degenerate in the sense that $\delta(p,q)=0$ does not imply $p\not= q$. On the other hand, the 
straight lines are geodesics for this metric and $\delta$ is translation invariant.
Minkowski geometry is the study of such weak metrics. It  plays an important role in convexity theory and in Finsler geometry where  Minkowski spaces play the role of flat spaces.

There is a vast literature on Minkowski metrics, and the goal of the present chapter is to provide for the reader of this handbook some  of the basic definitions and facts in the theory of weak Minkowski metrics, because of their relation to Hilbert geometry, and to give some examples. We give complete proofs of most of the stated results.  We end this chapter with  a  discussion about the relations and analogies between Minkowski geometry and Funk and Hilbert geometries.

\vspace{-0.2cm}

\section{Weak metric spaces}

We begin with the definition of a weak metric space.

\begin{definition}[Weak metric]
A {\it  weak metric}\index{weak metric}\index{metric!weak} on a set $X$ is a map $\delta:X\times X\to [0,\infty]$ satisfying the following two properties:
\begin{enumerate}
\item[a.)]  $\delta(x,x)=0$  for all $x$ in $X$;
\item[b.)]  $\delta ( x, y ) + \delta ( y, z ) \geq \delta ( x, z )$ for all $x$, $y$ and $z$ in $X$.
\end{enumerate}
\end{definition}
We often require a weak metric to satisfy somme additional properties. In particular one says that  
the weak metric $\delta$ on $X$ is
\begin{enumerate}
\item[c.)] {\it separating} if  $x \neq y$ implies $\delta ( x, y ) > 0$,
\item[d.)]  {\it weakly separating} if  $x \neq y$ implies $ \max{\{\delta(x,y),\delta(y,x)\}} > 0$,
\item[e.)]  {\it finite} if $\delta ( x, y ) < \infty$,
\item[f.)]  {\it reversible} (or {\it symmetric}) if $\delta (y,x) = \delta (x,y)$,
\item[g.)]  {\it quasi-reversible} if $\delta (y,x)\leq C \delta (x,y)$ for some constant $C$,
\end{enumerate}
for all $x$ and $y$ in $X$. 

One sometimes says that $\delta$ is {\it strongly separating} if condition (b)
holds, in order to stress the distinction with condition (d). Observe that for reversible metrics both notions of separation coincide.

\smallskip

A \emph{metric} in the classical sense is a reversible, finite and separating weak metric. Thus, it satisfies $$
 0 < \delta(x,y)=\delta(y,x)<\infty
$$
for all $x\not= y \in X$.

\begin{definition} \label{def.WPM} 
 Let $\mathcal{U} \subset X$ be a convex subset of a real vector space $X$. A weak metric $\delta$ in $\mathcal{U}$ is said to
 be \emph{projective}\index{projective weak metric}\index{weak metric!projective} (or \emph{projectively flat}){metric!projectively flat}{projectively flat!metric}  if satisfies the condition 
\begin{equation}\label{eq.projective}
 \delta ( x, y ) + \delta ( y, z ) = \delta ( x, z )
\end{equation}
whenever the three points $x$, $y$ and $z$ in $\mathcal{U}$ are aligned and $y \in [x,z]$, the affine segment from $x$ to $z$ (equivalently if $y = tx+(1-t)z$ for
some $0 \leq t \leq 1$).  The weak metric is \emph{strictly projective} if it is projective and 
$$
 \delta ( x, u ) + \delta ( u, z ) >  \delta ( x, z )
$$
whenever  $u \not\in [x,z]$.  
\end{definition}

\begin{example}\label{ex.PrFlat}
 A function $\varphi  : \r^n \to \r$ is said to be \emph{monotone increasing}  if whenever $x$ and $y$ in $\r^n$ satisfy 
 $x_j \leq y_j$ for $1 \leq j \leq n$ we have $\varphi (x) \leq \varphi (y)$. For any collection $\{ \varphi_a\}_{a\in A}$
 of monotone increasing functions defined on $\r^n$, the weak metric defined as
 $$
  \delta (x,y) = \sup_{a\in A} |\varphi(y) - \varphi(x)|
 $$
 is projective. A concrete example is given by
 $$
  \delta (x,y) = \max_{1 \leq j \leq n} |e^{y_j} - e^{x_j}|.
 $$
\end{example}

\begin{definition}[Weak Minkowski metric] \label{def.wm}
A {\it weak Minkowski metric}\index{weak Minkowski metric}\index{metric!weak Minkowski} on a real vector space $X$  is a weak metric $\delta$ on $X$ that is translation invariant and 
projective.
\end{definition} 

The weak metric in Example \ref{ex.PrFlat} is projectively flat, but in general it is not a weak Minkowski 
metric.

\begin{example}
 Let $X$ be a real vector space and $\varphi : X \to \r$ a linear form. Define 
 $\delta_{\varphi}(x,y) = \max\{0, \varphi (y-x)\}$. Then $\delta$ is a weak Minkowski metric. It is finite, but it is neither reversible nor weakly separating. 
\end{example}

\medskip

{\small 
In functional analysis, given a  real vector space $X$, the collection of sets 
$$
  B_{(\varphi,x,r)} = \{y \in X \tq \delta_{\varphi}(x,y) < r \} \subset X,
$$
where $x \in X$ is an arbitrary point, $r > 0$ and  $\varphi \in X^*$ is an arbitrary linear form generate a topology which is called the \emph{weak topology} 
on $X$. This observation is a possible justification for the name ``weak metric'' that we give to such functions. The terminology has its origin in the work of Ribeiro who was interested around 1943 \cite{Ribeiro} in some generalization of the Urysohn metrization theorem for the topology associated to $\delta$.
}

\medskip

\begin{example}[Counterexample]  Let $X$ be a real vector space and let $\| \, \|: X \to \r$ be a norm on $X$. 
Then $\delta (x,y) = \max\{\| y-x\|, 1\}$  is a metric that is translation invariant, but it is not a Minkowski metric because it is not projective. Indeed, suppose $\|z\| = 1$, then
$$
 \delta (0,2z) = 1 < \delta(0,z)+\delta(z,2z) = 2.
$$
\end{example}
 
In this example, the metric is ``projective for small distances'', in the sense that if $\|z-x\|\leq 1$ and $y\in [x,z]$, then (\ref{eq.projective}) holds. On the
other hand, large closed balls are not compact; in fact any ball of radius $\geq 1$ is equal to the whole space $X$.

\medskip

\begin{example}[Counterexample]  This is a variant of the previous example. Let again  $\| \, \|: X \to \r$ be a norm on the real vector space $X$.
Then $\rho_{\alpha} (x,y) =  \| y-x\|^{\alpha}$  is a metric if and only if $0 < \alpha \leq 1$. It is clearly  translation invariant, but it is not  projective
 if $\alpha < 1$, and thus it is not a Minkowski metric.
\end{example}
Unlike the previous metric $\delta$, the metric $\rho_{\alpha}$ is not projective for small distances (if $\alpha < 1$).  On the other hand, every closed ball
is compact.

\section{Weak Minkowski norms}

\begin{proposition}\label{proprieteF}
 Let $\delta$ be a weak Minkowski metric on some real vector space $X$ and set $F(x) = \delta (0,x)$.
 Then the function $F : X \to [0,\infty]$ satisfies the following properties
 \begin{enumerate}[i.)]
 	\item $F(x_1+x_2) \leq F(x_1)+ F(x_2)$ for all $x_1,x_2\in X $.
	\item  $F(\lambda x)=\lambda F(x)$ for all $x\in X$ and for all $\lambda\geq 0$.
\end{enumerate}
\end{proposition} 

\begin{proof}
The first property is a consequence of the triangle inequality together with the fact that $\delta$ is translation invariant:
 \begin{eqnarray*}
F(x+y) & = &  \delta (0,x+y)  
\\ &\leq &  \delta (0,x) +  \delta (x,x+y) \\
 & = & \delta (0,x) +  \delta (0,y)   \\
 & = & F(x) + F(y).
\end{eqnarray*}
To prove the second property, observe   for any  $x\in X$ and any $\lambda,   \mu \geq 0$ we have
$$
    \delta(0,\lambda x) + \delta(\lambda x,(\lambda+\mu) x) =    \delta(0,(\lambda+\mu) x),
$$
because $\lambda x$ belongs to the segment $[0, (\lambda+\mu) x]$. 
Since we have $\delta(\lambda x,(\lambda+\mu) x) = \delta(0,\mu x) =  F(\mu x)$, the previous identity can be written as
\begin{equation}\label{eq.pmon1}
   F(\lambda x) +  F(\mu x) = F((\mu +\lambda)x)
\end{equation}
and we conclude from the next lemma that $F(\lambda x)=\lambda F(x)$   for all $\lambda > 0$.
We also have $F(0\cdot x) = 0\cdot F(x) = 0$ since $F(0) = \delta (0,0) = 0$.
\end{proof}

\medskip

\begin{lemma}
 Let $f : \r_+ \to [0,\infty]$ be a function such that $f(\mu +\lambda) = f(\lambda) + f(\mu)$ for any $\lambda, \mu \in \r_+$,
 then 
 $$
   f(\lambda) = \lambda f(1)
 $$
 for every $\lambda >0$.
\end{lemma}

\begin{proof} We first assume $f(a) < \infty$ for every $a\in \r_+$. We have by hypothesis 
$$f(k\cdot a) = f(((k-1)+1)\cdot a) =  f((k-1)\cdot a) + f(a)$$
for any $k \in \mathbb{N} $.
We thus have by induction
$$
 f(k\cdot a) = k\cdot  f(a)
$$
for any $k \in \mathbb{N}$ and any $a\in \r_+$. 
Using the above identity with  $k,m\in \mathbb{N}$, we have
$$
 m \cdot f\left(\frac{k}{m}\right) =  f\left(m \frac{k}{m}\right) = f(k) =  k\cdot  f(1).
$$ 
Dividing this identity by $m$ we obtain 
$f(\alpha) = \alpha f(1)$ for any $\alpha \in \mathbb{Q}_+$.
Consider now  $\lambda \in \r_+$ arbitrary, and choose $\alpha_1, \alpha_2 \in \mathbb{Q}_+$ such that
$\alpha_1 < \lambda < \alpha_2$. Then 
$$
 f(\lambda) = f(\alpha_1) + f(\lambda - \alpha_1)> f(\alpha_1) =  \alpha_1f(1)
$$
and 
$$
    f(\lambda)  = f(\alpha_2)-f(\alpha_2 - \lambda) <f(\alpha_2)=\alpha_2f(1).
$$
Since $\alpha_2-\alpha_1 >0$ is arbitrarily small, we deduce that $ f(\lambda)  =  \lambda  f(1)$ for any $\lambda >0$.

\smallskip

So far we have assumed $f(a) < \infty$ for any $a>0$. Assume now there exists $a> 0$ such that $f(a) = \infty$. Then $f(\lambda)  =\infty$   for any $\lambda >0$. Indeed  choose an integer $k$
such that $k\lambda > a$. Then 
$$
   k f(\lambda) = f(k\lambda) = f(k\lambda - a) + f(a) \geq  f(a) = \infty.
$$
Therefore $f(\lambda) = f(1) = \infty$.
\end{proof}

\bigskip

\begin{definition} \label{def.WMN}
A function $F : X \to [0,\infty]$ defined on a real vector space $X$ is a \emph{weak Minkowski norm}
if the following two conditions hold:
 \begin{enumerate}[i.)]
 	\item $F(x_1+x_2) \leq F(x_1)+ F(x_2)$ for all $x_1,x_2\in X $.
	\item  $F(\lambda x)=\lambda F(x)$ for all $x\in X$ and for all $\lambda\geq 0$.
\end{enumerate}
\end{definition}

Proposition \ref{proprieteF} states that a weak Minkowski metric determines a weak Minkowski norm. 
Conversely, a weak Minkowski 
norm defines a weak Minkowski metric  $\delta_F$ by the formula 
\begin{equation}\label{eq.deltaF}
 \delta_F (x,y) = F(y-x).
\end{equation}

We then naturally define a weak Minkowski $F$ norm to be
\begin{enumerate}[\ $\circ$]
  \item \emph{separating} if $x\neq 0$ implies $F(x) > 0$;
  \item \emph{weakly separating} if $x\neq 0$ implies $\max\{F(x),F(-x)\} > 0$;
  \item \emph{finite} if $F(x) < \infty$;
  \item \emph{reversible} (or \emph{symmetric}) if $F(-x) = F(x)$;
\end{enumerate}
for any $x\in X$.
 
\medskip

\begin{example}
The function $F : \r^2 \to [0, \infty]$ defined by $F(x_1,x_2) = \max\{x_1,0\}$ if $x_2=0$ and $F(x_1,x_2) =\infty$ if $x_2\neq 0$ is a weak Minkowski norm which is neither finite, nor separating, nor symmetric. It is however weakly separating.
\end{example} 
 
 \medskip
 
Observe that in this example $F$ is finite on some vector subspace of $\r^2$. This is a general fact:
 
\begin{proposition}\label{prop.continuity}
Let  $F : X \to [0,\infty]$ be a {weak Minkowski norm} on the real vector space $X$ and set $D_F = \{x\in X : F(x) < \infty\}$.
Then $D_F$ is a vector subspace of $X$.
Furthermore,  the restriction of $F$ to any finite-dimensional subspace  $E \subset D_F$ is continuous.
\end{proposition} 
 
\begin{proof}  
If $x,y \in D_F$, then $F(x)$ and $F(y)$ are finite and therefore $F(x+y) \leq F(x) + F(y) < \infty$
and $F(\lambda x) = \lambda F(x) < \infty$. Therefore $x+y \in D_F$ and $\lambda x \in D_F$, which 
proves the first assertion.

To prove the second part, we consider a finite-dimensional subspace  $E \subset D_F$
and we choose a basis $e_1,e_2,\dots, e_m \in E$. Define the constant
$$
  C = \max_{1 \leq j \leq m} (F(e_j) + F(-e_j)).
$$
For an arbitrary vector $x = \sum_{j=1}^m \alpha_j e_j \in E$ we then have
$$
 F(x) \leq  \sum_{j=1}^m F(\alpha_j e_j)  \leq C  \cdot \sum_{j=1}^m |\alpha_j |.
$$
In particular if $x \to 0$, then $F(x) \to 0$. More generally, if a sequence $x_{\nu} \in E$ converges to
some $a \in E$, then 
$$
 \limsup_{\nu \to \infty} F(x_{\nu}) =  \limsup_{\nu \to \infty} F(a+(x_{\nu}-a)) \leq
 F(a) + \limsup_{\nu \to \infty} F((x_{\nu}-a)) = F(a).
$$
Since $F(a) \leq F(x_{\nu}) +  F(a-x_{\nu})$ we also have
$$F(a) \leq \liminf_{\nu \to \infty} (F(x_{\nu}) +   F(a-x_{\nu})) = \liminf_{\nu \to \infty} F(x_{\nu}).$$
It follows that
$$
  \limsup_{\nu \to \infty} F(x_{\nu})   \leq  F(a) \leq  \liminf_{\nu \to \infty} F(x_{\nu}),
$$
and the continuity on $E$ follows.
\end{proof}

\medskip 

\begin{corollary}\label{cor.semicontinuity}
 Any weak Minkowski norm on a finite-dimensional vector space $X$ is lower semi-continous.
\end{corollary}

\begin{proof}
 We need to prove that $F(a) \leq  \liminf_{\nu \to \infty} F(x_{\nu})$ for every sequence $x_{\nu} \in X$ converging to
 $a$. If $F(a) = \infty$, then $a$ belongs to the open set $X \setminus D_F$. It follows then that $x_{\nu} \not\in D_F$
 for large enough $\nu$ and therefore
 $$
   \liminf_{\nu \to \infty} F(x_{\nu}) = \infty = F(a).
 $$
 If  $F(a) < \infty$, then two cases may occur. If infinitely many $x_{\nu}$ belong to $D_F$, then by the previous proposition we have
 $$
  F(a) = \lim_{\nu \to \infty, x_{\nu} \in D_F} F(x_{\nu}) = \liminf_{\nu \to \infty} F(x_{\nu}).
 $$
 If on the other hand $D_F$ contains only finitely many elements of the sequence $x_{\nu}$, then
 $$
   \liminf_{\nu \to \infty} F(x_{\nu}) = \infty > F(a).
 $$ 
\end{proof}

\medskip

\begin{definition}
Given a weak Minkowski norm $F$ on a vector space $X$ we define the open and closed \emph{unit balls} at the origin as
$$
 \Omega_F = \{ x \in X : F(x) < 1\}, \quad \text{and} \quad  \overline{\Omega}_F = \{ x \in X : F(x) \leq 1\}.
$$
The set 
$$
  \mathcal{I}_F = \{ x \in X : F(x) = 1\} 
$$
is called the \emph{unit sphere} or the \emph{indicatrix} of $F$. \index{Indicatrix}
\end{definition}

\begin{proposition}\label{prop.int_cont}
 Let $F$ be a weak Minkowski norm on a finite-dimensional vector space $X$. Then the following are equivalent
 \begin{enumerate}
  \item $F$ is finite,
  \item $F$ is continuous,
  \item $\Omega_F$ is open,
  \item $0$ is an interior point of $\Omega_F$.
\end{enumerate}
\end{proposition}

\begin{proof}  The implication (1) $\Rightarrow$ (2) is Proposition \ref{prop.continuity} and the implications 
(2) $\Rightarrow$ (3)  $\Rightarrow$ (4) are obvious. To prove 
(4) $\Rightarrow$ (1) we suppose $F$ is not finite. Then there exists $a\in X$ such that $F(a) = \infty$. Thus, $F(\lambda a)
= \infty$ for all $\lambda >0$, in particular $\lambda a \not\in \Omega_F$  for all $\lambda >0$ and therefore 
$0$ is not an interior point of $\Omega_F$.
\end{proof}

\medskip

\begin{proposition}\label{prop.Fseparates}
 Let $F$ be a weak Minkowski norm on $\r^n$. Then the following are equivalent
 \begin{enumerate}
  \item $F$ is separating (i.e. $F(x) >0$ for all $x\neq 0$),
  \item $F$ is bounded below on the Euclidean unit sphere $S^{n-1} \subset \mathcal{U}$,
  \item $\overline{\Omega}_F$ is bounded.
\end{enumerate}
\end{proposition}

\begin{proof} 
 (1) $\Rightarrow$ (2):  Suppose that $F$ is not bounded below on $S^{n-1}$. Then there
exists a sequence $x_j \in S^{n-1}$ such that $F(x_j) \to 0$. Choosing a subsequence if necessary, 
we may assume, by compactness of the sphere, that $F(x_j) < \infty$ for all $j$, i.e. $x_j \in D_F \cap S^{n-1}$ and that $x_j$ 
converges to some point $x_0 \in D_F \cap S^{n-1}$. 
Since $F$ is continuous on $D_F$, we have $F(x_0) = \lim_{j \to \infty}F(x_j) = 0$. Since
$x_0 \neq 0$ (it is a point on the sphere), it follows that $F$ is not separating. \\
(2) $\Rightarrow$ (3):  Condition (2) states that there exists $\mu >0$ such that $F(x) \geq \mu$ for all $x \in S^{n-1}$. Therefore $F(y) \leq 1$ implies  $\|y\| \leq  \frac{1}{\mu}$.\\
(3) $\Rightarrow$ (1):   Suppose $F$ is non separating. Then there exists $x \neq 0$ with $F(x) = 0$.
Therefore $F(\lambda x) = 0$ for any $\lambda >0$. In particular $\r_+ x \subset \overline{\Omega}_F$ which is therefore
unbounded.
\end{proof}

\medskip

\begin{definition}
 A \emph{Minkowski norm} is a weak Minkowski norm that is finite and separating. It is simply called a \emph{norm}
  if it is furthermore reversible. 
\end{definition}

To a finite and separating norm is associated a well-defined topology, viz. the topology associated to the symmetrization of the weak metric defined by Equation (\ref{eq.deltaF}) (which is a genuine metric).
For a deeper investigation of various topological questions we refer to the book \cite{Cobzas} by  S. Cobzas.

\begin{corollary}
 The topology defined by the distance (\ref{eq.deltaF}) associated to a Minkowski norm on  $\r^n$ 
 coincides with the Euclidean topology. 
\end{corollary}

\begin{proof}
Proposition  \ref{prop.continuity} implies that $F$ is continuous. From the compactness of 
the Euclidean unit sphere $S^{n-1}$ we thus have a constant $\mu>0$ 
such that $\mu \leq F(x) \leq \frac{1}{\mu}$ for all points $x$ on  $S^{n-1}$. It follows that 
\begin{equation}\label{eq.estnorm}
  \mu \|x\| \leq F(x) \leq \frac{1}{\mu}\|x\|
\end{equation}
 for all $x\in \r^n$ and therefore $F$ induces the same topology as the Euclidean norm.
\end{proof}

\medskip

The next result shows how one can reconstruct the weak Minkowski norm from its unit ball.

\medskip

\begin{proposition}
 Let $\Omega \subset \r^n$ be a convex set containing the origin. Define a  function $F : \r^n \to [0,\infty]$
 by
\begin{equation} \label{mink.functional}
  F(x) = \inf \{ t \geq 0 : x \in t\cdot \Omega \}.
\end{equation} 
Then $F$ is a weak Minkowski norm and $\overline{\Omega}_F$ coincides with the closure of $\Omega$, that is
 $
  \overline{\Omega} =  \{ x \in \r^n : F(x) \leq 1 \}.
 $
 Furthermore, if $\Omega$ is open, then $\Omega =  \{ x \in \r^n : F(x) < 1 \}$.
\end{proposition} 

\medskip

The function $F$ defined by (\ref{mink.functional}) is called the \emph{Minkowski functional} of $\Omega$.
\index{Minkowski functional}\index{functional!Minkowski}

\medskip

\begin{proof}
 We need to verify the two conditions in Definition \ref{def.WMN}. For $\lambda>0$, we have
  \begin{eqnarray*}
   F(\lambda x)  &= &\inf \{ s \geq 0 : \lambda x \in s\cdot \Omega \}
     \\    & =  &\inf \{ s \geq 0 :  x \in \frac{s}{\lambda}\cdot \Omega \}
 \\     &\underset{(s=\lambda t)}{=}   &\lambda \inf \{ t \geq 0 :  x \in t\cdot \Omega \}
 \\ &  = & \lambda F(x).
\end{eqnarray*}
Now because $\Omega$ is convex we have for $s,t > 0$
$$
 \frac{x}{s} \in \Omega \mbox{ and } \frac{y}{s} \in \Omega \ \Longrightarrow \ 
 \frac{x+y}{s+t} =  \frac{s\cdot \frac{x}{s}+t \cdot \frac{y}{t}}{s+t} \in \Omega.
$$
Therefore
$$
 F(x) < s \mbox{ and }F(y) < t \ \Longrightarrow \
 F(x+y) < s+t,
$$
which is equivalent to $F(x+y) \leq F(x) + F(y)$. his proves the first part of the proposition. \\
To prove the remaining assertions, observe that $F(x) \leq 1$ means that $tx \in \Omega$ for 
any $0<t<1$ and thus $x \in \overline{\Omega}$. This shows that 
$$
   {\Omega} \subset  \{ x \in \r^n \tq  F(x) \leq 1 \}   \subset \overline{\Omega}.
$$
The converse inclusion $\overline{\Omega} \subset  \{ x \in \r^n \tq F(x) \leq 1 \}$ follows from the lower semi-continuity
of $F$ (Corollary \ref{cor.semicontinuity}). Finally, if $\Omega$ is open then $F$ is continuous (Proposition
\ref{prop.int_cont}) and therefore $ {\Omega} =  \{ x \in \r^n : F(x) < 1 \}$.
\end{proof}

\medskip

Thus we have established one-to-one correspondences between weak Min--kowski metrics on $\r^n$,  weak Minkowski norms and closed convex sets containing the origin. The closed convex set associated to a weak Minkowski norm $F$ is the set $\overline{\Omega}_F= \{ x \in X : F(x) \leq 1\}
$. The associated weak metric is separating if and only if the associated convex set is bounded and the metric is finite if and only if the origin is an interior point of the convex set.

\medskip

\begin{remark}
These concepts have some important consequences in convex geometry. For instance one can easily prove that \emph{every unbounded convex set in $\r^n$ must contain a ray}. Indeed, let $\Omega \subset \r^n$ be unbounded and convex. One may assume that $\Omega$ contains the origin. Then by Proposition \ref{prop.Fseparates} 
its  weak Minkowski functional $F$ is not separating, that is, there exists $a\neq 0$ in $\r^n$
such that $F(a) = 0$; but then $F(\lambda a) = 0$ for every $\lambda >0$ and therefore the ray $\Omega$ contains the ray  $\r_+a$.
\end{remark}

\medskip  

Let us conclude this section with two important results from Minkowski geometry. A Minkowski norm on $\r^n$ is said to be 
\emph{ Euclidean} if it is associated to a scalar product.

\begin{proposition}\label{prop.round}
Let $\delta$ be a  Minkowski metric on $\r^n$. Then $\delta$ is a Euclidean metric if and only the ball
$$
  B_{(a,r)} = \{x \in X \tq \delta(a,x) < r \} \subset \r^n
$$
(for some arbitrary $a\in \r^n$ and $r>0$) is an ellipsoid centered at $a$.
\end{proposition}

Notice that the above proposition is false if the ellipsoid is not centered at $a$.


\begin{proof}
Recall that by definition an (open) ellipsoid is a convex set in $\r^n$ that is the affine image of the open Euclidean unit ball.
If the weak metric $\delta$ is Euclidean, then it is obvious that every ball is an ellipsoid. Conversely, suppose that some ball
of an arbitrary  Minkowski metric $\delta$ is an ellipsoid. Then the ball with the same radius centered at the origin is also an ellipsoid since $\delta$ is translation invariant, that is,  $B_{(0,r)} =\{x \in X : F(x)=\delta(0,x) < r \}$ is an ellipsoid. But then
$$
 \Omega = B_{(0,1)} = \frac{1}{r} \cdot B_{(0,1)}
$$
is also an ellipsoid. Changing coordinates if necessary, one may assume that 
$$
  \Omega = \{x \in \r^n \tq \sum_i x_i^2 < 1\},
$$ 
which is the Euclidean unit ball. It follows that
$$
  F(x) = \inf \{ t>0 \tq  x \in t\,  \Omega \} = \inf \{ t>0 \tq    \|x\| < t\} = \|x\|
$$
where $\| \cdot \|$ denotes the Euclidean norm. 
\end{proof}

We have the following result on the isometries of a Minkowski metric.

\begin{theorem}
Let $\delta$ be a  Minkowski metric on $\r^n$. Then every isometry of $\delta$ is an affine transformation of $\r^n$,
and the group $\Iso(\r^n,\delta)$ of isometries of $\delta$ is conjugate within the affine group to a subgroup of the group $E(n)$
of Euclidean isometries of $\r^n$.
Furthermore $\Iso(\r^n,\delta)$ is conjugate to the full group $E(n)$ if and only if $\delta$ is a Euclidean metric.
\end{theorem}

\begin{proof}
The first assertion is the Mazur-Ulam Theorem, see \cite{Nica}\index{Mazur-Ulam Theorem}.\index{Theorem!Mazur-Ulam}. To prove the 
second assertion, we recall that every bounded convex set $\Omega$ in $\r^n$ with non-empty interior contains a unique ellipsoid $J_{\Omega} \subset \Omega$ of maximal volume. This is called the \emph{John ellipsoid}\index{John ellipsoid}\index{ellipsoid!John} of $\Omega$, see \cite{Barvinok}. 

Let us consider the unit ball $\Omega = B_{(\delta, 0,1)}$ of our Minkowski metric and let us denote by $J$ its John ellipsoid
and by $z\in J$ its center. We call $J^* = J-z$ the \emph{centered John ellipsoid}\index{centered John ellipsoid}\index{ellipsoid!centered John} of $\Omega$. 
Consider now an arbitrary isometry $g \in \Iso(\r^n,\delta)$. Set $\widetilde{g}(x) = g(x) - b$, where $b = g(0)$.
Then $\widetilde{g}$ is 
an  isometry for $\delta$ fixing the origin. By construction and uniqueness, the centered John ellipsoid is invariant: $\widetilde{g}(J^*) = J^*$. There exists an element $A \in \mathrm{GL}_n(\r)$ such that $AJ^* = \mathbb{B}$ is the Euclidean unit ball.  Let us set $f:= A\circ \widetilde{g} \circ A^{-1}$. Then
$$
 f (\mathbb{B}) =  \mathbb{B}.
$$
By Mazur-Ulam $\widetilde{g}$ is a linear map, therefore $f$ is a linear map preserving the Euclidean unit ball, which means that $f\in O(n)$. We thus obtain
$$
  g(x) = A^{-1}(f(x) + Ab) A
$$
where $A$ is linear and $x \mapsto f(x) + Ab$ is a Euclidean isometry.

To prove the last assertion, one may assume, changing coordinates if necessary, that $\Iso(\r^n, \delta) = E(n)$.
Then the $\delta$-unit ball $\Omega$ is invariant under the orthogonal group $O(n)$ and it is therefore a round sphere.
We now conclude from Proposition \ref{prop.round} that $\delta$ is Euclidean.
\end{proof}

\section{The midpoint property}

\begin{definition}
 A weak metric $\delta$ on the real vector space $X$ satisfies the \emph{midpoint property}\index{midpoint property} if for any $p,q \in X$ we have
 $$
     \delta (p,m) = \delta (m,q) = \frac{1}{2}\delta (p,q)
 $$
 where $m = \frac{1}{2}(p+q)$ is the affine midpoint of $p$ and $q$.
\end{definition}

To describe the main consequence of this property, we shall use the notion of dyadic numbers.

\begin{definition}
  A \emph{dyadic number} is a rational number of the type $\lambda = 2^{-k}m$ with $m,k \in \mathbb{Z}$.
  We  denote  the set of dyadic numbers by
  $$\mathbb{D} = \bigcup_{k =0}^{\infty} 2^{-k} \mathbb{Z},$$
  and  the subset of nonnegative dyadic numbers by
  $\mathbb{D}_+ \subset \mathbb{D}.$\index{dydadic number}
\end{definition}

\begin{proposition}\label{prop.distpr}
 Let $\delta$ be a  weak metric on the real vector space $X$.
 Then $\delta$  satisfies the midpoint property if and only if for any pair of distinct points  $p,q \in X$ and  for any  
 $\mu,\lambda$ in $\mathbb{D}$ with  $\mu \leq \lambda$, we have
 \begin{equation} \label{eq.dydprop}
     \delta(\gamma(\mu),\gamma(\lambda)) = (\lambda - \mu)\cdot  \delta (p,q)
 \end{equation}
 where $\gamma (t) = tp + (1-t)q$.
\end{proposition}

\begin{proof}  It is obvious that if (\ref{eq.dydprop}) holds then $\delta$ satisfies the midpoint property.
The proof of the other direction requires several steps. Assume that $\delta$ satisfies the midpoint property. Then we have
$$
  \delta(p,\gamma(\tfrac{1}{2})) = \frac{1}{2}\delta (p,q)  \qquad \text{and} \qquad
  \delta(p,\gamma(2) = 2\delta (p,q).
$$
By an induction argument we then have
\begin{equation}\label{eq.px1}
  \delta(p,\gamma(2^m)) = 2^m\delta (p,q)
\end{equation}
for any $k \in \mathbb{N}$. Because $p$ is the midpoint of $\gamma(-2^m)$ and  $\gamma(2^m)$,
we deduce that
\begin{equation}\label{eq.px1}
  \delta(\gamma(-2^m),\gamma(2^m)) = 2^{m+1}\delta (p,q).
\end{equation}
Now we have for $k\in \mathbb{Z}$
$$
  \delta(\gamma(k-1),\gamma(k)) =  \delta(\gamma(k),\gamma(k+1)) 
$$
since $\gamma(k)$ is the midpoint of $\gamma(k-1)$ and $\gamma(k+1)$.
Because $\delta (p,q) =  \delta(\gamma(0),\gamma(1))$, we deduce that
$$
  \delta(\gamma(k),\gamma(k+1)) = \delta (p,q),
$$
and by the triangle inequality we have
\begin{equation}\label{eq.px2}
  \delta(\gamma(i),\gamma(j)) \leq  (j-i)\delta (p,q)
\end{equation}
for any $i,j \in \mathbb{Z}$ with $i<j$. We will show that this inequality is in fact an equality.
Choose $m\in \mathbb{N}$ with $2^m \geq \max (|i|,|j|))$. Then we have from (\ref{eq.px1}) and 
 (\ref{eq.px2}) 
 \begin{eqnarray*}
2^{m+1}\delta(p,q) &=&  \delta(\gamma(-2^{m}), \gamma(2^m)) 
\\ &\leq&
  \delta(\gamma(-2^{m}), \gamma(i)) +  \delta(\gamma(i), \gamma(j))
  + \delta(\gamma(j), \gamma(2^m)).
\end{eqnarray*}

Using now (\ref{eq.px2}) we have $\delta(\gamma(i),\gamma(j)) \leq  (j-i)\delta (p,q)$, but also
$$
  \delta(\gamma(-2^{m}), \gamma(i))   \leq (i+2^{m}) \delta (p,q),
$$
and
$$
 \delta(\gamma(j), \gamma(2^m))  \leq   (2^{m}-j) \delta (p,q).
$$
Since 
$$
  (i+2^{k}) + (j-i) +(2^{m}-j)  = 2^{m+1}
$$
all the above inequalities must be equalities. Thus, we have established that
\begin{equation}\label{eq.px3}
 \delta(\gamma(i), \gamma(j))  = (j-i) \delta(p,q)
\end{equation}
for any $i,j  \in \mathbb{Z}$.

\medskip

Let us now fix $k \in \mathbb{N}$ and set $q_k = \gamma(2^{-k})$ and 
$$
  \gamma_k(t) = \gamma(t2^{-k}) = tp + (1-t)q_k.
$$
Applying  (\ref{eq.px3}) to $\gamma_k$ we have 
$$
 \delta(\gamma_k(i), \gamma_k(j))  = (j-i) \delta(p,q_k) = (j-i)2^{-k}\delta(p,q).
$$ 
The latter  can be rewritten as
 $$
   \delta(\gamma(\tfrac{i}{2^k}), \gamma(\tfrac{j}{2^k})) = (\tfrac{j}{2^k}-\tfrac{i}{2^k})\delta(p,q)
 $$
 for any $i,j  \in \mathbb{Z}$ and $k \in \mathbb{N}$, which  is equivalent to (\ref{eq.dydprop})
 for any dyadic number  $\mu,\lambda$  with  $\mu \leq \lambda$. 
 \end{proof}
    
The next result is a generalization to the case of weak metrics of a characterization of
Minkowski geometry due to  H. Busemann in \cite[\S 17]{Busemann1955}.

\begin{theorem}
 A finite weak metric $\delta$ on $\r^n$ is a weak Minkowski metric if and only if it satisfies 
 the midpoint property and if its restriction to every  affine line is continuous. More precisely, the latter condition means that if $a$ and $b$ are two points in $\r^n$, then for any $t_0 \in\r$ we have
 $$
  \lim_{t \to t_0}\delta(\gamma (t), b)) =  \delta(\gamma (t_0), b))
 $$
 and 
 $$
  \lim_{t \to t_0}\delta(a,\gamma (t)) =  \delta(a,\gamma (t_0)),
 $$
 where $\gamma (t) = ta + (1-t)b$.
\end{theorem}
 
\begin{proof} 
 If $\delta$ is a Minkowski metric, then it is projective and since $\delta$ is finite (by hypothesis),  it follows
 from  Propositions   \ref{proprieteF} and \ref{prop.continuity} that the distance is given by
 $$
   \delta(x,y) = F(y-x),
 $$
 where $F$ is a weak Minkowski norm. The continuity of $\delta$ follows now from  Proposition \ref{prop.continuity} 
 and the midpoint property follows from property (ii) in Proposition \ref{proprieteF}.

\smallskip

Conversely , let us assume that the weak metric $\delta$ satisfies the midpoint property and that it is continuous on every line.  We need to show that $\delta$ is  projective and translation invariant.
 
 We first observe that \emph{if $a,b \in \r^n$ are two distinct points with $\delta (a,b) \neq 0$
 and if $x$ and $y$ are two points aligned with 
 $a$ and $b$ such that  $(y-x)$ is a nonnegative multiple of $(b-a)$, then 
\begin{equation}\label{eq.distratio}
   \frac{\delta(x,y)}{\delta(a,b)} = \frac{|y-x|}{|b-a|},
\end{equation}
 where $|q-p|$ denotes the Euclidean distance between $p$ and $q$ in $\r^n$}. This follows   from 
 Proposition \ref{prop.distpr}, together with the continuity of $\delta$ on lines and the density of $\mathbb{D}$
 in $\r$.

 This immediately implies that  $\delta(p, z) +  \delta(z,q) =\delta(p,q)$ whenever $z  \in [p,q]$, meaning that the weak metric $\delta$ is projective.

 \smallskip

To prove the translation invariance, we consider four points $p,q,p',q'$ with $(q'-p') = (q-p)$.
If the four points are on a line, then (\ref{eq.distratio}) implies that $\delta(p',q') = \delta(p,q)$.
If  the four points are not on a line, then $pqq'p'$ is a non-degenerate  parallelogram.
 Assume also $0<\delta(p,q) < \infty$ and denote by $L^+_{pq}$ the ray
 with origin $p$ through $q$ and $L_{qq'}$ the line passing through  $q$ and  $q'$. 
 Choose a sequence $y_j \in L^+_{pq}$ such
 that $|y_j-p| \to \infty$ and  set  $x_j = L_{p'y_j} \cap L^+_{qq'}$. 

\bigskip

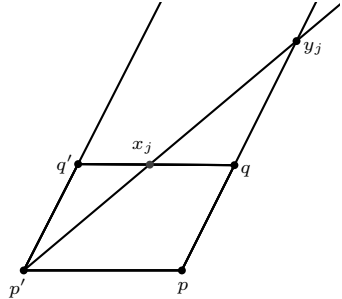
\psset{xunit=0.7cm,yunit=0.7cm,algebraic=true,dotstyle=o,dotsize=3pt 0,linewidth=0.8pt,arrowsize=3pt 2,arrowinset=0.25}
\begin{pspicture*}(-4,-1.0)(6,5.1)
\pspolygon(0,0)(3,0)(4,2)(1.03,2.02)
\psline(0,0)(3,0)
\psline(3,0)(4,2)
\psline(4,2)(1.03,2.02)
\psline(1.03,2.02)(0,0)
\psplot{0}{7.4}{(-0--2.02*x)/1.03}
\psplot{3}{7.4}{(-6--2*x)/1}
\psplot{0}{7.4}{(-0--4.36*x)/5.18}
\psline(1.03,2.02)(4,2)
\psline(0,0)(3,0)
\begin{scriptsize}
\psdots[dotstyle=*](0,0)
\rput(-0.11,-0.33){$p'$}
\psdots[dotstyle=*](3,0)
\rput(3.02,-0.33){$p$}
\psdots[dotstyle=*](4,2)
\rput(4.21,1.9){$q$}
\psdots[dotstyle=*](1.03,2.02)
\rput(0.78,2.03){$q'$}
\psdots[dotstyle=*](5.18,4.36)
\rput(5.5,4.22){$y_j$}
\psdots[dotstyle=*,linecolor=darkgray](2.39,2.01)
\rput(2.26,2.34){$x_j$}
\end{scriptsize}
\end{pspicture*}

We then have
 $$
  1- \frac{\delta(p,p')}{\delta(p,y_j)} =  \frac{\delta(p,y_j) - \delta(p,p')}{\delta(p,y_j)} \leq  
  \frac{\delta(p',y_j)}{\delta(p,y_j)}    \leq   \frac{\delta(p',p) + \delta(p,y_j)}{\delta(p,y_j)} 
  = \frac{\delta(p',p)}{\delta(p,y_j)} + 1.
 $$
Using (\ref{eq.distratio}), we have  $\delta(p,y_j) \to \infty$, therefore
$$
      \lim_{j \to \infty} \frac{\delta(p',x_j)}{\delta(p,q)} 
     = \lim_{j \to \infty} \frac{\delta(p',y_j)}{\delta(p,y_j)} = 1.
$$ 
Because $x_j \to q'$ on the line $L_{qq'}$, we have by hypothesiss
$$
    \lim_{j \to \infty} \delta(x_j,q') = \lim_{j \to \infty} \delta(q',x_j) =  0,
$$ 
and since
$$
 \delta (p',q') - \delta(x_j,q') \leq \delta(p',x_j) \leq  \delta (p',q') +\lim_{j \to \infty} \delta(q',x_j),
$$ 
we have  $\delta(p',x_j) \to \delta(p',q')$. Therefore
  $$
     \frac{\delta(p',q')}{\delta(p,q)} =\lim_{j \to \infty} \frac{\delta(p',x_j)}{\delta(p,q)} = 1.
  $$  
It follows that for a nondegenerate parallelogram $p,q,q',p'$, we have $\delta(p',q') = \delta(p,q)$.

Suppose now that $\delta(p,q) = 0$. Then we also have $\delta(p',q') = 0$ for otherwise, exchanging the role
of $p,q$ and $p',q'$ in the previous argument, we get a contradiction.

We thus have established that in all cases  $\delta(p',q')  = \delta(p,q)$  if  $q'-p' = q-p$. In other words,  
$\delta$ is translation invariant. Since it is projective, this completes the proof that it is a weak Minkowski metric.
\end{proof}

 \begin{example}[Counterexample]
 Let $X$ a be real vector space and let $h : X \to \r$ be an injective  $\mathbb{Q}$-linear map.  
 Then the function $\delta : X \times X \to \r$ defined by
 $$
  \delta (x,y) = |h(x)-h(y)|
 $$
 is a  metric which is translation invariant and satisfies the midpoint property.
 Yet it is in general not projective    (unless  $h$ is  $\r$-linear, and thus $\dim_{\r}(X) = 1$).
\end{example}

\section{Strictly and strongly convex Minkowski norms}

\begin{definition}
(i) Let $F$ be a (finite and separating) Minkowski norm in $\r^n$ with unit ball $\Omega_F$. Then $F$ is said to be
\emph{strictly convex}\index{strictly convex Minkowski norm}\index{Minkowski norm!strictly convex}  if  the indicatrix $\partial \Omega_F$ contains no non trivial segment, that is, if for any $p,q \in \partial \Omega_F$, we have 
$$
  [p,q] \subset \partial \Omega \, \Rightarrow \  p=q.
$$
(ii) The function $F$ is said to be \emph{strongly convex}\index{strongly convex Minkowski norm}\index{Minkowski norm!strongly convex}   if $F$ is smooth on $\r^n \setminus \{ 0\}$ and the hypersurface $\partial \Omega_F \subset \r^n$
has everywhere positive Gaussian curvature. Equivalently, the Hessian 
\begin{equation}\label{eq.hessianM}
\g_{y}(\eta_1,\eta_2) = \frac{1}{2}\left.\frac{\partial^2}{\partial u_1\partial u_2}\right|_{u_1=u_2=0} F^2(y+u_1\eta_1+u_2\eta_2)
\end{equation}
 of $F^2(y)$ is positive definite for any point $y \in \r^n \setminus \{ 0\}$. 
\end{definition}

There are several equivalent definitions of strict convexity in Minkowski spaces, see e.g. \cite{Day,Minkowski}.

It is clear that a strongly convex Minkowski norm is strictly convex. The converse does not hold: the $L_p$-norm
$$
 \| y \|_p = \left( \sum_{j=1}^n |y_j|^p\right)^{1/p}
$$
is an example of a smooth strictly convex norm which is not strongly convex.

\begin{proposition}
 Let $F$  be a strongly convex Minkowski norm on $\r^n$. Then $F$ can be recovered from its Hessian via the formula
\begin{equation}\label{eq.Fgg}
   F(y) = \sqrt{\g_{y}(y,y)}
\end{equation}
 where $\g_y$ is  defined by (\ref{eq.hessianM}). 
\end{proposition}

\medskip

This result follows from applying twice  the following Lemma, which is sometimes  called the \emph{Euler Lemma}.\index{Euler lemma}\index{lemma!Euler}

\begin{lemma}\label{lem.euler}
Let $\psi : \r\setminus {0}\to \r$ be a positively homogeneous functions of degree $r$. If $\psi$ is of class $C^k$ for some $k\geq 1$, then the partial derivatives $\frac{\partial \psi}{\partial y^i}$ are positively homogenous functions of degree $r-1$ and
$$
  r\cdot \psi(y) = \sum_{i=1}^n y^i\frac{\partial \psi}{\partial y^i}.
$$
In particular $y^i\frac{\partial \psi}{\partial y^i} = 0$ if $\psi$ is $0$-homogenous.
\end{lemma}

Recall that a function $\psi : \r^n \setminus {0}\to \r$ is
\emph{positively homogenous} of degree $r$ if $\psi(\lambda y) = \lambda^r \psi(y)$ for all $y \in  \r^n\setminus {0}$ and all $\lambda >0$.

\begin{proof}
This is elementary: we just differentiate the function $t \mapsto \psi(ty) = t^r\psi(y)$ to obtain
$$
  \frac{\partial \psi}{\partial y^i}(ty)\cdot y^i =  rt^{r-1}\cdot \psi(y),
$$
and set $t=1$.
\end{proof}

If $F$ is a strongly convex Minkowski norm on $\r^n$, then Formula (\ref{eq.hessianM}) defines a Riemannian metric  $\g_y$
on $\r^n \setminus \{ 0\}$. Using Lemma \ref{lem.euler}, on gets that $\g_y$ is invariant under homothety,
that is we have $\g_{\lambda y} = \g_y$  for every $\lambda >0$ and $y \in \r^n \setminus \{0\}$. Furthermore $F$ is
determined from this metric  by Equation  (\ref{eq.Fgg}).  
We conclude from these remarks the following:

\begin{proposition}
There is a natural bijection between strongly convex Minkowski norms
on $\r^n$ and  Riemannians metric on $\r^n \setminus \{ 0\}$ which are invariant under homothety.
\end{proposition}

This observation can be used as a founding stone for Minkowski geometry, see e.g. \cite{Varga}, and it plays a central role in Finsler Geometry.

\section{The synthetic viewpoint}

Definition \ref{def.wm} of a  weak Minkowski space is based on a real vector space $X$ as a 
ground space. In fact only the affine structure of that space plays a role and we could equivalently 
start with a given affine space instead of a vector space.

The synthetic viewpoint is to start with an abstract metric space and to try to give a list of natural conditions 
implying  the given metric space to be Minkowskian. This question, and similar questions for other geometries, has been a central and recurring question in the work H. Busemann, and it is implicit in Hilbert's comments on his Fourth Problem \cite{Hilbert-Problems}. Some answers are given in his book 
\emph{The Geometry of Geodesics} \cite{Busemann1955}, in that book Busemann introduces the notions of
\emph{$G$-spaces} and  \emph{Desarguesian spaces}. The goal of this section is to give a short account on this viewpoint. We restrict ourselves to the case of ordinary metric spaces.

\begin{definition}[Busemann $G$-space]
A \emph{Busemann $G$-space}\index{Busemann $G$-space}\index{space!Busemann $G$-} is a metric space $(X,d)$, satisfying the following four conditions:
\begin{enumerate}
  \item (Menger Convexity) Given distinct points $x, y \in X$ , there is a point $z \in X$ different from $x$ and $y$ such that $d(x, z) + d(z, y) = d(x, y)$.
  \item (Finite Compactness) Every $d$-bounded infinite set has an accumulation point.
  \item  (Local Extendibility) For every point $p\in X$ , there exists $r_p>0$, such that for any pair of distinct points $x, y\in X$ in the open ball $B(p,r_p)$, there is a point $z\in B(p,r_p)\setminus \{x,y\}$ such that $d(x,y)+d(y,z)=d(x,z)$.
    \item  (Uniqueness of Extension) Let  $x, y,z_1,z_2$ be four points in $X$ such that
    $d(x,y)+d(y,z_1)=d(x,z_1)$ and  $d(x,y)+d(y,z_2)=d(x,z_2)$.
    Suppose that  $d(y, z_1) = d(y, z_2)$, then $z_1 = z_2$.
\end{enumerate}
\end{definition}
A typical example of a Busemann $G$-space $(X,d)$ is a strongly convex Finsler manifold of class $C^2$  (and in fact of class $C^{1,1}$ by a result of Pogorelov). It follows from the definition that  any pair of points in a Busemann $G$-space $(X,d)$ 
can be joined by a minimal geodesic and that geodesics are locally unique. It is also known
that every $G$-space is topologically homogeneous and that it is a manifold if its dimension 
is at most 4. We refer to \cite{BHR} for further results on the topology of $G$-spaces.

\medskip

Among $G$-spaces, Busemann introduced the class of Desarguesian spaces.

\begin{definition}[Desarguesian space]  A \emph{Desarguesian space}\index{Desarguesian space}\index{space!Desarguesian} is a metric space $(X,d)$ satisfying the following conditions:
\begin{enumerate}
  \item $(X,d)$ is a  a Busemann $G$-space. 
  \item $(X,d)$ is uniquely geodesic, that is every pair of points can be joined by a unique geodesic.
  \item If the topological dimension\footnote{On page 46  in \cite{Busemann1955}, Busemann states that he is using the Menger-Urysohn notion
of dimension, but any reasonable notion of topological dimension is equivalent for a $G$-space.}
 of $X$ equals 2, then Desargues theorem holds for the family of all geodesics.
  \item If  the topological dimension of $X$ is greater than 2, then any triple of points lie in a plane, 
  that is, a two-dimensional subspace of $X$ which is itself a $G$-space.
\end{enumerate}
\end{definition}
The reason for assuming Desargues' property  in the 2-dimensional case as an axiom is due to the well known  
fact from axiomatic geometry that it is possible to construct exotic 2-dimensional objects satisfying the 
axioms of real projective or affine geometry but which are not isomorphic to $\mathbb{RP}^2$ or
$\mathbb{R}^2$ (an example of such exotic object is the Moufang plane); these objects do not satisfy Desargues property.  Similar objects do not exist in higher dimension and Desargues property is in fact a theorem in all dimensions $\geq 3$. 
Condition (3) in the above definition could be rephrased as follows: \emph{If $X$ is 2-dimensional, then it can be isometrically embedded in a 3-dimensional Desarguesian space.}
We refer to \cite{Busemann1955} and \cite{Papadopoulos-Hilbert} for further discussion of Desarguesian spaces.

\medskip

A deep result of Busemann states that a Desarguesian space can be mapped on a real projective space or on a
convex domain in a real affine space with a projective metric. More precisely he proved the following

\begin{theorem}[Theorems 13.1 and 14.1  in \cite{Busemann1955}] \label{th.RepDesargues}
Given an $n$-dimensional  Desarguesian space $(X,d)$, one of the following condition holds:
\begin{enumerate}
  \item Either all the geodesics are topological circles and there is a homeomorphism $\varphi : X \to \mathbb{RP}^n$ that maps every  geodesic in $X$ onto a projective line;
  \item or there is a homeomorphism from  $X$ onto a convex domain $\mathcal{C}$ in $\r^n$  that maps every  geodesic   in $X$ onto the intersection of a straight line with $\mathcal{C}$.
\end{enumerate}
\end{theorem}

Using the notion of Desarguesian space and following Busemann, we now give two purely intrinsic characterizations of finite-dimensional  Minkowski spaces among abstract metric spaces.
Note  that a Minkowski space $(X,d)$ is a $G-$space  if and only if its unit ball is strictly convex. The first result is
a converse to that statement.

  \begin{theorem}[\cite{Busemann1955}, Theorem 24.1]  \label{th,dsgr1}  
  A metric space $(X.d)$ is isometric to a Minkowski space if and only if it is a 
  Desarguesian space in which the parallel postulate holds and the spheres are stirctly convex.
 \end{theorem}
 
Observe that in a Desarguesian space there  are well defined notions of lines and planes and therefore
Euclid's parallel  can be formulated. Using Theorem \ref{th.RepDesargues} and the parallel postulate, we  obtain that  $(X,d)$ is isometric to $\r^n$ with a projectively flat metric. To prove the Theorem, Busemann 
uses the strict convexity of spheres to establish the midpoint property.

\medskip

The next result we state involves the notion of \emph{Busemann zero curvature}.\index{Busemann zero curvature}\index{zero curvature!Busemann} Recall that a geodesic metric space
is said to have \emph{zero curvature in the sense of Busemann}, if the distance between the midpoints of two sides of an arbitrary triangle is equal to half the length of the remaining side. Busemann then formulates the following characterization:
\begin{theorem}[\cite{Busemann1955}, Theorem 39.12]  \label{th,dsgr2}     
  A  simply connected finite-dimensional $G-$space of zero curvature is  isometric to a Minkowski space.  
 \end{theorem}

 \medskip
 
Busemann came back several times to the problem of characterizing Min\-kowskian and locally Minkowskian spaces. In his paper with Phadke \cite{BP1979}, written 25 years after \cite{Busemann1955}, he gave sufficient conditions that are more technical but weaker than those of Theorem \ref{th,dsgr2}.

\section{Comparison and analogies between Minkowski gaometry and Funk and Hilbert geometries}

Given a Minkowski metric $\delta$ in $\r^n$ whose unit ball $\Omega$ at the origin is open and bounded,
the distance between two points is obtained by setting $\delta(x,x)=0$ for all $x$ in $\mathbb{R}^n$ and, for $x\not=y$, 
$$\displaystyle \delta(x,y)=\frac{\vert x-y\vert}{\vert 0-a^+\vert}$$
 where $\vert \ \vert$ denotes the Euclidean metric and the point  $a^+$ is the intersection with $\partial \Omega$ of  the ray starting at the origin $0$ of $\mathbb{R}^n$ and parallel to the ray $R(x,y)$ from $x$ to $y$. This formula is equivalent to (\ref{def.mk}) and it  suggest an analogy with the formula for the Funk distance in the domain $\Omega$ (see Definition 2.1 in the chapter \cite{PT_Funk}  of this volume). It is also in the spirit of the following definition of Busemann (\cite{Busemann1955}, Definition 17.1): \emph{A metric $d(x,y)$ in $\r^n$ is Minkowskian if for the euclidean metric $e(x,y)$ the distances $d(x,y)$ and $e(x,y)$ are proportional on each line.} 
   
 Minkowski metrics share several important properties of the Funk and the Hilbert metrics, and it is interesting to compare these classes of metrics. Let us quickly review some of the analogies. 
 
We start by recalling that in the formulation of Hilbert's fourth problem which asks for the construction and the study of metrics on subsets of Euclidean (or of projective) space for which the Euclidean segments are geodesics, the Minkowski and Hilbert metrics appear together as the two examples that Hilbert gives (see \cite{Hilbert-Problems} and the chapter \cite{Papadopoulos-Hilbert} in this volume).

A rather simple analogy between the Minkowski and the Funk geometries is that both metrics are uniquely geodesic if and only if their associated convex sets are strictly convex. (Here, the convex set associated to a Minkowski metric is the unit ball centered at the origin. The convex set associated to a Funk metric is the set on which this metric is defined.)

Another analogy between Minkowski and Hilbert geometries is the well known fact that a Minkowski weak metric on $\mathbb{R}^n$ is Riemannian if and only if the associated convex set is an ellipsoid, see Proposition \ref{prop.round}. This fact is (at least formally) analogous to the fact that the Hilbert geometry of an open bounded convex subset of $\mathbb{R}^n$ is Riemannian if and only if the convex set is an ellipsoid (see 
 \cite{Kay} Proposition 19).
  
As a further relation between Minkowski and Hilbert geometries, let us recall a result obtained by Nussbaum, de la Harpe, Foertsch and Karlsson. Nussbaum and de la Harpe proved (independently) in \cite{Nussbaum1988} and \cite{Harpe} that if $\Omega\subset\mathbb{R}^n$ is the interior of the standard $n$-simplex and if $H_{\Omega}$ denotes the associated Hilbert metric, then the metric space $(\Omega,H_{\Omega})$ is isometric to a Minkowski metric space. Foertsch and Karlsson  proved the converse in \cite{FK}, thus completing the proof of the fact that a bounded open convex subset $\Omega$ of $\mathbb{R}^n$ equipped with its Hilbert metric is isometric to a Minkowski space if and only if $\Omega$ is the interior of a simplex.

It should be noted that the result (in both directions) was already known to Busemann since 1967. In their paper \cite{BP1984}, p. 313, Busemann and Phadke write the following, concerning the simplex: 

\begin{quote}\small
The case of general dimension $n$ is most interesting. The (unique) Hilbert geometry possessing a transitive abelian group of motions where the affine segments are the chords (motion means that both distance and chords are preserved) is given by a simplex $S$, (\cite{B1967} p. 35). If we realize $\mathcal{I}$ [the interior of the simplex] as the first quadrant $x_i>0$ of an affine coordinate system, the group is given by $x'_i=\beta_i x_i$, $\beta_i >0$ [...] $m$ is a Minkowski metric because it is invariant under the translations and we can take the affine segments as chords". 
\end{quote}

We finally mention the following  common characterizations of Minkowski-Funk geometries
and of Minkowski-Hilbert geometries:
 
\begin{theorem}[Busemann  \cite{Busemann1970}]\label{Busemann-iso}
 Among noncompact and nonnecessarily symmetric Desarguesian space in which all the right and left spheres of positive radius around any point are compact, the Hilbert and Minkowski geometries are characterized by the property that any isometry between two (distinct or not) geodesics is a projectivity. 
 \end{theorem}

 \begin{theorem}[Busemann  \cite{Busemann-homothetic}]\label{Busemann-homot}
 A Desarguesian space in which all the right spheres of positive radius around any point are homothetic is either a Funk space or a Minkowski space.
 \end{theorem}


\printindex
 
\end{document}